\documentclass[graybox]{svmult}

\usepackage{amsmath,amssymb,theorem,color}

\usepackage{mathptmx}       % selects Times Roman as basic font
\usepackage{helvet}         % selects Helvetica as sans-serif font
\usepackage{courier}        % selects Courier as typewriter font
\usepackage{type1cm}        % activate if the above 3 fonts are
\usepackage{makeidx}         % allows index generation
\usepackage{graphicx}        % standard LaTeX graphics tool
\usepackage{multicol}        % used for the two-column index
\usepackage[bottom]{footmisc}% places footnotes at page bottom

\newcommand{\R}{\ensuremath{\mathbb{R}}}
\newcommand{\N}{\ensuremath{\mathbb{N}}}

\newcommand{\cF}{\mathcal{F}}

\newcommand{\mP}{\mathbb{P}}
\newcommand{\ltn}{\ensuremath{\left| \! \left| \! \left|}}
\newcommand{\rtn}{\ensuremath{\right| \! \right| \! \right|}}

\makeindex             % used for the subject index
\begin{document}

\title*{Young differential delay equations driven by H\"older continuous paths}
\author{Luu Hoang Duc and Phan Thanh Hong}
\institute{Luu Hoang Duc  \at Max Planck Institute for Mathematics in the Sciences, 04103 Leipzig, Germany, \& \\ Institute of Mathematics, Vietnam Academy of Science and Technology, 10307 Hanoi, Vietnam, \email{ duc.luu@mis.mpg.de, lhduc@math.ac.vn}
\and Phan Thanh Hong \at Thang Long University, 128200 Hanoi, Vietnam, \email{hongpt@thanglong.edu.vn }}
\maketitle

\abstract*{Each chapter should be preceded by an abstract (10--15 lines long) that summarizes the content. The abstract will appear \textit{online} at \url{www.SpringerLink.com} and be available with unrestricted access. This allows unregistered users to read the abstract as a teaser for the complete chapter. As a general rule the abstracts will not appear in the printed version of your book unless it is the style of your particular book or that of the series to which your book belongs.
Please use the 'starred' version of the new Springer \texttt{abstract} command for typesetting the text of the online abstracts (cf. source file of this chapter template \texttt{abstract}) and include them with the source files of your manuscript. Use the plain \texttt{abstract} command if the abstract is also to appear in the printed version of the book.}

\abstract{In this paper we prove the existence and uniqueness of the solution of Young differential delay equations under weaker conditions than it is known in the literature. We also prove the continuity and differentiability of the solution with respect to the initial function and give an estimate for the growth of the solution. The proofs use techniques of stopping times, Shauder-Tychonoff fixed point theorem and a Gronwall-type lemma. }

\section{Introduction}

In this paper we would like to study the deterministic delay equation of the differential form
\begin{eqnarray}\label{Eq1}
dx(t)&=&f(x_t)dt+g(x_t)d\omega(t), \qquad t \in [0,T] \\
x_0&=&\eta\in C_r: = C([-r,0],\R^d) \notag
\end{eqnarray}
or in the integral form
\begin{eqnarray}\label{Eq1int}
x(t)&=&x(0)+ \int_0^t f(x_s)ds+ \int_0^t g(x_s)d\omega(s), \qquad t\in [0,T] \\
x_0&=&\eta\in C_r\notag
\end{eqnarray}
for some fixed time interval $[0,T]$, where $C([a,b],\R^d)$ denote the space of all continuous paths $x:\;[a,b] \to \R^d$ equipped with sup norm $\|x\|_{\infty,[a,b]}=\sup_{t\in [a,b]} \|x(t)\|$,  with $\|\cdot\|$ is the Euclidean norm in $\R^d$, $x_t\in C_r$ is defined by $x_t(u):=x(t+u)$ for all $u\in [-r,0]$; $f,g: C_r\to\R^d$ are coefficient functions; and $\omega$ belongs to $C^{\nu\rm{-Hol}}([0,T],\R)$ - the space of H\"older continuous paths for index $\nu > \frac{1}{2}$. Such system appears, for example, while solving stochastic differential equations of the form
\begin{equation}\label{fSDDE}
dx(t)=f(x_t)dt+g(x_t)dB^H(t), \qquad x_0=\eta\in C_r,
\end{equation}
where $B^H$ is a fractional Brownian motion defined on a complete probability space $(\Omega,\cF,\mP)$ with the Hurst index $H \in (1/2,1)$ \cite{mandelbrot}. Since $B^H$ is neither Markov nor semimartingale if $H \ne \frac{1}{2}$, one cannot apply the classical Ito theory to solve \eqref{fSDDE}. Instead, due to the fact that $B^H(\cdot)$ is H\"older continuous for almost surely all the realizations, one can define the stochastic integral w.r.t.\ the fBm as the integral driven by a H\"older continuous path using the so called {\it rough path theory} \cite{friz}, \cite{lyons98}, \cite{lyonsqian}, \cite{lyonsetal}, or {\it fractional calculus theory} \cite{samko}, \cite{zahle}. As a result, solving \eqref{fSDDE} leads to the deterministic equation \eqref{Eq1} or \eqref{Eq1int}, where the second integral in \eqref{Eq1int} is understood in the Young sense (see \cite{lyons94},  \cite{young}). 

The theory of stochastic differential equations driven by the fBm $B^H$ for $H>\frac{1}{2}$ has been well developed by many authors, especially results on existence and uniqueness of the pathwise solution, the generation of random dynamical systems (see e.g.  \cite{chen}, \cite{congduchong}, \cite{atienza2},  \cite{lejay}, \cite{lyons94},   \cite{nualart2}, \cite{nualart3}, \cite{young},... and the references therein). For studies on delay equations, we refer to \cite{bou1}, \cite{bou2}, \cite{bou3},  \cite{duc-siegmund-schmalfuss}.

In the general case where $f,g$ are functions of $(t,x_t)$, under some regularity conditions, i.e. $f$ is globally Lipschitz continuous and of linear growth, $g$  is $C^1$ such that its Frechet derivative is bounded and globally Lipschitz continuous, there exists a unique solution $x(\cdot,\omega,\eta)$ of \eqref{Eq1} (see \cite{bou1} or \cite{shev}). These results are based on the tools of fractional calculus developed in \cite{nualart3}, \cite{zahle}, \cite{zahle2}. \\
In this paper, we reprove the existence and uniqueness theorem of \eqref{Eq1} under the following assumptions.

(${\textbf H}_f$) The function $f$ is globally Lipschitz continuous and thus has linear growth, i.e  there exist constants $L_f$ such that for all $\xi,\eta\in C_r$
\[
\|f(\xi)-f(\eta)\|\leq L_f\|\xi-\eta\|_{\infty,[-r,0]} 
\]

(${\textbf H}_g$) The function $g$ is $C^1$ such that its Frechet derivative is bounded and locally $\delta-$H\"older continuous with $1\geq \delta >\frac{1-\nu}{\nu}$, i.e there exists $L_g$ such that for all $\xi,\eta\in C_r$
\[
\|Dg(\xi)\|_{L(C_r,\R^d)}\leq L_g 
\]
and for each $M>0$, there exists $L_M$ such that for all $\xi,\eta\in C_r$ that satisfy 
\[
\|\xi\|_{\infty,[-r,0]},\|\eta\|_{\infty,[-r,0]}\leq M
\] 
one has
\begin{equation}\label{Holderiv}
\|Dg(\xi)- Dg(\eta)\|_{L(C_r,\R^d)}\leq L_M\|\xi-\eta\|^\delta_{\infty,[-r,0]}
\end{equation}
for some constant $1>\delta >\frac{1-\nu}{\nu}$. Assumption \eqref{Holderiv} is weaker than the global Lipschitz continuity of $Dg$, as seen in \cite{bou1}, \cite{duc-siegmund-schmalfuss} or \cite{shev}. \\
Furthermore, we show that the solution is differentiable with respect to the initial function $\eta$ and give an estimate for the growth of the solution. Note that in order to define the second integral in \eqref{Eq1int} in the Young sense, one needs to consider the solution $x$ and the initial function $\eta$ in H\"older function spaces $C^{\beta\text{-Hol}}$ with $\beta +\nu >1$. 

To finish the introduction, we recall some facts about Young integral, more details can be seen in \cite{friz}. For $p\geq 1$ and $[a,b] \subset \R$, a continuous path $x:[a,b] \to \R^d$ is of finite $p-$variation if 
\begin{eqnarray}
\ltn x\rtn_{p\text{-var},[a,b]} :=\left(\sup_{\Pi(a,b)}\sum_{i=1}^n\|x(t_{i+1})-x(t_i)\|^p\right)^{1/p} < \infty,
\end{eqnarray}
where the supremum is taken over the whole class of finite partition of $[a,b]$. The subspace $C^{p-\text{var}}([a,b],\R^d)\subset C([a,b],\R^d)$ of all paths  $x$ with finite $p-$variation and equipped with the $p-$var norm
\begin{eqnarray*}
	\|x\|_{p\text{-var},[a,b]}&:=& \|x(a)\|+\ltn x\rtn_{p\text{-var},[a,b]},
\end{eqnarray*}
is a nonseparable Banach space \cite[Theorem 5.25, p.\ 92]{friz}.\\
Also, for $0<\beta \leq 1$ denote by $C^{\beta\rm{-Hol} }([a,b],\R^d)$ the Banach space of all H\"older continuous paths $x:[a,b]\to \R^d$ with exponent $\beta$, equipped with the norm
\begin{eqnarray}
\|x\|_{\infty,\beta,[a,b]} &:=& \|x\|_{\infty,[a,b]} + \ltn x\rtn_{\beta,[a,b]}\ \text{where}\notag\\
\ltn x\rtn_{\beta,[a,b]}&:=&  \sup_{a\leq s<t\leq b} \frac{\|x(t)-x(s)\|}{(t-s)^{\beta}}< \infty.
\end{eqnarray}
Note that the space is not separable. However, the closure of $C^{\infty}([a,b],\R^d)$ in the $\beta-$ Holder norm denoted by $C^{0,\beta\rm{-Hol}}([a,b],\R^d)$ is a separable space (see \cite[Theorem 5.31, p. 96]{friz}), which can be defined as
\[
C^{0,\beta\rm{-Hol}}([a,b],\R^d) := \Big\{x \in C^{\beta\text{-Hol}}([a,b],\R^d) \Big|\lim \limits_{h \to 0} \sup_{a\leq s<t\leq b, |t-s|\leq h} \frac{\|x(t)-x(s)\|}{(t-s)^{\beta}} = 0 \Big\}.
\]
Clearly, if $x\in C^{\beta\rm{-Hol}\ }([a,b],\R^d)$ then for all $s,t\in [a,b]$ we have
\begin{eqnarray*}%\label{pvar-Hol.norm}
	\|x(t)-x(s)\|&\leq& \ltn x\rtn_{\beta,[a,b]} |t-s|^\beta.
\end{eqnarray*}
%Due to \cite[Proposition 5.10, p. 83]{friz},
Hence, for all $p$ such that $p\beta\geq 1$ we have
\begin{eqnarray}\label{pvar-Hol.norm}
\ltn x\rtn_{p\text{-var},[a,b]}&\leq& \ltn x\rtn_{\beta,[a,b]} (b-a)^\beta<\infty.
\end{eqnarray} 
In particular, $C^{1/p\text{-Hol} }([a,b],\R^d)\subset C^{p-\text{var}}([a,b],\R^d)$. \\
For $a,b,c\in\R$ such that $a<b<c$ and $x \in C^{\beta\text{-Hol}}([a,c],\R^d)$,  it is easy to see that 
$$
\ltn x\rtn_{\beta,[a,c]} \leq \ltn x\rtn_{\beta,[a,b]}+\ltn x\rtn_{\beta,[b,c]}.
$$
Now consider $x \in C^{\beta\text{-Hol}}([a,b],\R^d)$ and $\omega \in C^{\nu\text{-Hol}}([a,b],\R)$ with $\beta+\nu>1$. Then by \eqref{pvar-Hol.norm}, $x \in C^{\frac{1}{\beta}\text{-var}}([a,b],\R^d)$ and $\omega \in C^{\frac{1}{\nu}\text{-var}}([a,b],\R)$, thus it is well known that the Young integral $\int_a^bx(t)d\omega(t)$ exists (see \cite[p.\ 264-265]{friz}). Moreover, for all $s\leq t$ in $[a,b]$, due to the Young-Loeve estimate \cite[Theorem 6.8, p.~116]{friz}
\begin{eqnarray*}
	\left\|\int_s^tx(u)d\omega(u) - x(s)[\omega(t)-\omega(s)]\right\| &\leq& K \ltn \omega \rtn_{\frac{1}{\nu}\text{-var},[s,t]} \ltn x \rtn_{\frac{1}{\beta}\text{-var},[s,t]} \\
	&\leq& K (t-s)^{\beta+\nu}\ltn \omega \rtn_{\nu,[s,t]} \ltn x \rtn_{\beta,[s,t]},
\end{eqnarray*}
where $K:= \frac{1}{1-2^{1-(\beta+\nu)}}$. Hence
\begin{equation}
\left\|\int_s^tx(u)d\omega(u)\right\| \leq (t-s)^\nu\ltn\omega\rtn_{\nu,[s,t]} \left(\|x(s)\|+K(t-s)^{\beta}\ltn x\rtn_{\beta,[s,t]} \right).
\end{equation}

\section{Existence, uniqueness and continuity of the solution}

Since $\delta \nu + \nu >1$, there exists $\beta < \nu$ such that
\[
\beta+\nu > \beta\delta +\nu>1.
\]
By choosing a smaller $\nu' \in (\frac{1}{2},\nu)$ if necessary, we can always assume without loss of generality that $\omega\in C^{0,\nu\rm{-Hol}}([0,T],\R)$. System \eqref{Eq1} would then be considered for $\eta \in C^{\beta\rm{-Hol}}([-r,0],\R^d)$, i.e. we consider the equation 
\begin{eqnarray}\label{Eq2}
dx(t)&=&f(x_t)dt+g(x_t)d\omega(t), \qquad t\in [0,T]\\
x_0&=&\eta\in C^{\beta\rm{-Hol}}([-r,0],\R^d). \notag
\end{eqnarray}

%=========================
\begin{lemma}\label{lem1}
	If $x\in C^{\beta\rm{-Hol}}([a-r,b],\R^d)$ then the function $x_{.}:[a,b]\rightarrow C_r$, $x_t(\cdot)= x(t+\cdot)$ belongs to $C^{\beta\rm{-Hol}}([a,b],C_r)$ and satisfies
	\begin{eqnarray}
	&i,& \ltn x_{.} \rtn_{\beta,[a,b]} \leq \ltn x \rtn_{\beta,[a-r,b]} \label{lem1eq1}\\
	&ii,&\| x_{.} \|_{\infty,\beta,[a,b]} \leq \| x \|_{\infty,\beta,[a-r,b]}. \label{lem1eq2}
	\end{eqnarray}
\end{lemma}
\begin{proof}
	The fact that
	\begin{eqnarray*}
		\ltn x_{.} \rtn_{\beta,[a,b]} &=& \sup_{a\leq s<t \leq b}\frac{\|x_t-x_s\|_{\infty,[-r,0]}}{(t-s)^\beta} \\
		&=&\sup_{a\leq s<t \leq b}\quad\sup_{-r\leq u\leq 0}\frac{\|x(t+u)-x(s+u)\|}{[(t+u)-(s+u)]^\beta}\notag\\
		&\leq & \sup_{a-r\leq s'<t'\leq b}\frac{\|x(t')-x(s')\|}{(t'-s')^\beta}=\ltn x \rtn_{\beta,[a-r,b]} \\
	\end{eqnarray*}
	proves \eqref{lem1eq1}. As a result,
	\begin{eqnarray*}
		\| x_{.} \|_{\infty,\beta,[a,b]}&=& \sup_{t\in [a,b]}\|x_t\|_{\infty,[-r,0]}+\ltn x_{.} \rtn_{\beta,[a,b]}\\
		& \leq& \|x(\cdot)\|_{\infty,[a-r,b]}+ \ltn x \rtn_{\beta,[a-r,b]},
	\end{eqnarray*}
	which proves \eqref{lem1eq2}.
\end{proof}

\begin{remark}
	The lemma is not true if we replace the H\"older continuous space by $p-$variation bounded space. Namely, if a function $x$ belongs to $C^{p\rm{-var}}([a-r,b],\R^d)$, it does not follow that its translation function $x_\cdot$ belongs to $C^{p\rm{-var}}([a,b],C_r)$  with $p\geq 1$. As a counterexample, consider the function $x(t)=|t|^\beta$, $t\in [-1,1]$, $\beta p <1$ then $x\in C^{p\rm{-var}}([-1,1],\R)$. However,
	with the partition $\Pi = 0< \frac{1}{n}<\frac{2}{n}<\dots < \frac{n-1}{n} < 1$ we have
	\begin{eqnarray*}
		\left(\sum_i \|x_{\frac{i+1}{n}} -x_{\frac{i}{n}} \|^{p}_{\infty,[-1,0]}\right)^{1/p} &=&\left(\sum_i \sup_{-1\leq u\leq 0}\Big|x(\frac{i+1}{n}+u) -x(\frac{i}{n}+u) \Big|^{p}\right)^{1/p}\\
		&\geq & \left(\sum_i \Big|x(\frac{i+1}{n}-\frac{i}{n} ) -x(\frac{i}{n}-\frac{i}{n}) \Big|^{p}\right)^{1/p}\\
		&\geq & \left(\sum_i \frac{1}{n^{\beta p}}\right)^{1/p}\\
		&=&n^{\frac{1-\beta p}{p}} \to \infty,\;\; \text{as} \;\; n\to \infty.
	\end{eqnarray*}
	This shows that $x_.$ is not of bounded $p-$variation.
\end{remark}

%=========================
\begin{lemma}\label{lem2}
	Assume that $g$ satisfies the condition $({\textbf H}_g)$. If $x\in C^{\beta\rm{-Hol}}([a-r,b],\R^d)$  then $g(x_.)\in C^{\beta\rm{-Hol}}([a,b],\R^d)$ and 
	\begin{eqnarray}
	\ltn g(x_{.} )\rtn_{\beta,[a,b]} \leq L_g\ltn x \rtn_{\beta,[a-r,b]}.
	\end{eqnarray}
\end{lemma}
\begin{proof}
	The proof is directed from the Lipschitz continuity of $g$ and Lemma \ref{lem1}. Namely, %for $s<t$ in $[a,b]$
	\begin{eqnarray*}
		\sup_{a\leq s<t\leq b}\frac{\|g(x_t)-g(x_s)\|}{(t-s)^{\beta}} &\leq& \sup_{a\leq s<t\leq b} L_g\frac{\|x_t-x_s\|_{\infty,[-r,0]}}{(t-s)^{\beta}}\\
		&\leq & L_g \ltn x \rtn_{\beta,[a-r,b]}.
	\end{eqnarray*}
\end{proof}

%=========================

\begin{remark}
	Since $\beta+\nu>1$ the integral $\int_a^bg(x_t)d\omega(t)$ is well defined.
\end{remark}

%=========================

\begin{lemma}\label{lem3}
	Assume that $g$ satisfies the condition $({\textbf H}_g)$. If $x,y\in C^{\beta\rm{-Hol}}([a-r,b],\R^d)$ are such that $\|x\|_{\infty,\beta,[a-r,b]},\|y\|_{\infty,\beta,[a-r,b]}\leq M$, then 
	\begin{eqnarray}\label{lem3eq1}
	\ltn g(x_.)-g(y_.)\rtn_{\delta\beta,[a,b]} &\leq& L_g(b-a)^{\beta-\delta\beta}\ltn x-y\rtn_{\beta,[a-r,b]}+L_M M^{\delta}\|x-y\|_{\infty,[a-r,b]}\notag\\
	&\leq& \Big(L_g(b-a)^{\beta-\delta\beta} + L_M M^{\delta}\Big) \|x-y\|_{\infty,\beta,[a-r,b]}
	\end{eqnarray}
\end{lemma}
%============
\begin{proof}
	%Fix $s<t$ in $[a,b]$, by mean value theorem we have (\textcolor{red}{Hong: check this theorem for Banach space case})
	By the mean value theorem 
	\begin{eqnarray*}
		&&|g(x_t)-g(y_t) - g(x_s)+g(y_s)|\\
		&=& \left| \int_0^1Dg(\theta x_t + (1-\theta y_t))(x_t-y_t)d\theta+\int_0^1Dg(\theta x_s + (1-\theta) y_s)(x_s-y_s)d\theta\right|\\
		&\leq&  \left| \int_0^1Dg(\theta x_t + (1-\theta) y_t)[(x_t-y_t)-(x_s-y_s)]d\theta\right|\\
		&& +\left|\int_0^1[Dg(\theta x_t + (1-\theta) y_t)-Dg(\theta x_s + (1-\theta) y_s)](x_s-y_s)d\theta\right|\\
		&\leq&  L_g \|(x_t-y_t)-(x_s-y_s)\|_{\infty,[-r,0]}\\
		&&+ L_M \|x_s-y_s\|_{\infty,[-r,0]}\int_0^1 \left(\theta \|x_t-x_s\|^\delta_{\infty,[-r,0]}+(1-\theta)\|y_t-y_s\|^\delta_{\infty,[-r,0]}\right) d\theta\\
		&\leq&  L_g (t-s)^\beta\ltn x_.-y_.\rtn_{\beta,[a,b]} \\
		&&+ L_M\|x_s-y_s\|_{\infty,[-r,0]} (t-s)^{\delta\beta}\max \left\{\ltn x_.\rtn^\delta_{\beta,[a,b]},\ltn y_.\rtn^\delta_{\beta,[a,b]}\right\}\\
		&\leq&   L_g (t-s)^\beta\ltn x-y\rtn_{\beta,[a-r,b]} + L_M(t-s)^{\delta\beta}M^{\delta}\|x-y\|_{\infty,[a-r,b]}.
	\end{eqnarray*}
	This implies 
	\begin{eqnarray*}
		\ltn g(x_.)-g(y_.)\rtn_{\delta\beta,[a,b]}
		&\leq & L_g(b-a)^{\beta-\delta\beta}\ltn x-y\rtn_{\beta,[a-r,b]}
		+ L_M M^\delta\| x-y\|_{\infty,[a-r,b]}.
	\end{eqnarray*}
\end{proof}
%===========================================
Consider $x\in C^{\beta\rm{-Hol}}([t_0-r,t_1],\R^d)$ with any interval $[t_0,t_1] \subset[0,T]$. 
Put
$$
I(x)(t): = \int_{t_0}^t f(x_s) ds\;\; \text{and}\;\; J(x)(t):= \int_{t_0}^t g(x_s)d\omega(s),\;\; t\in [t_0,t_1]
$$
and define the map
\[F(x)(t)=
\begin{cases}
&  x(t_0) + I(x)(t)+J(x)(t)\quad \text{if}\;\; t\in [t_0,t_1]\\
& x(t) \qquad \qquad \qquad \qquad \qquad \text{if}\;\; t\in [t_0-r,t_0]
\end{cases}
\]
%\textcolor{red}{Hong: norm $\|.\|_{\infty,\beta-Hol}$ might be more convenient!}
%================================
\begin{lemma}\label{lem4}
	If $x,y\in C^{\beta\rm{-Hol}}([a-r,b],\R^d)$ are such that $\|x\|_{\infty,\beta,[a-r,b]}, \|y\|_{\infty,\beta,[a-r,b]}\leq M$, then there exists $L=L(b-a,M)$ satisfying
	\begin{equation}\label{lem4eq1}
	\ltn F(x)-F(y)\rtn_{\beta,[a,b]}\leq  L\left((b-a)^{1-\beta}+ (b-a)^{\nu-\beta}\ltn\omega\rtn_{\nu,[a,b]}\right) \| x-y\|_{\infty,\beta,[a-r,b]}.
	\end{equation}
	%$x(t)=y(t), \forall t\in [a-r,a]$ and
\end{lemma}
\begin{proof}
	First, observe that% for all $s<t$ in $[a,b]$
	\begin{eqnarray}\label{lem4eq2}
	\ltn I(x)-I(y)\rtn_{\beta,[a,b]}&=&
	\sup_{a\leq s<t\leq b}\frac{|I(x)(t)-I(y)(t)-I(x)(s)+I(y)(s)|}{(t-s)^\beta}\notag\\
	&\leq &\sup_{a\leq s<t\leq b}\frac{\int_s^t|f(x_u)-f(y_u)|du}{(t-s)^\beta}\notag\\
	&\leq& \sup_{a\leq s<t\leq b}\frac{L_f(t-s)\|x-y\|_{\infty,[a-r,b]}}{(t-s)^\beta}\notag\\
	&\leq& L_f(b-a)^{1-\beta}\|x-y\|_{\infty,[a-r,b]}.
	\end{eqnarray}
	Secondly, since $\nu+\delta\beta>1$, by assigning $K'=\frac{1}{1-2^{1-(\nu+\delta\beta)}}$ and applying Lemma \ref{lem3} one has
	\begin{eqnarray}\label{lem4eq3}
	&& \sup_{a\leq s<t\leq b}\frac{|J(x)(t)-J(y)(t)-J(x)(s)+J(y)(s)|}{(t-s)^\beta} \notag\\
	&\leq &\sup_{a\leq s<t\leq b}\frac{|\int_s^t[g(x_u)-g(y_u)]d\omega(u)|}{(t-s)^\beta}\notag\\
	&\leq& \sup_{a\leq s<t\leq b}\frac{(t-s)^{\nu}\ltn\omega\rtn_{\nu,[s,t]}\left[\|g(x_s)-g(y_s)\|+K'(t-s)^{\delta\beta}\ltn g(x_.)-g(y_.)\rtn_{\delta\beta,[s,t]}\right]}{(t-s)^\beta}\notag\\
	&\leq& (b-a)^{\nu-\beta}\ltn\omega\rtn_{\nu,[a,b]}  \left[ L_g\|x-y\|_{\infty,[a-r,b]}+L_gK'(b-a)^{\beta}\ltn x-y\rtn_{\beta,[a-r,b]}\right.\notag\\
	&&\qquad\qquad\qquad\qquad\qquad\qquad\qquad\left.+K'L_MM^{\delta}(b-a)^{\delta\beta}\|x-y\|_{\infty,[a-r,b]}\right].
	\end{eqnarray}
	Now \eqref{lem4eq1} is followed from \eqref{lem4eq2} and \eqref{lem4eq3} by choosing
	\begin{equation}\label{eqL}
	L = L(b-a,M) := L_f + L_g + L_g K' (b-a)^\beta +  K'L_M M^{\delta} (b-a)^{\delta\beta}.
	\end{equation}
\end{proof}
%==================

We can now state the theorem on existence and uniqueness of solution of system \eqref{Eq1}.

%==================

\begin{theorem}\label{thm1}
	Assume that $({\textbf H}_f)$ and $({\textbf H}_g)$ are satisfied. If $\eta \in C^{\beta\rm{-Hol}}([-r,0],\R^d)$ then there exists a unique solution to the equation \eqref{Eq2} in $C^{\beta\rm{-Hol}}([-r,T],\R^d)$. Moreover, the solution is $\nu-$H\"older continuous on $[0,T]$.
\end{theorem}
%==================
\begin{proof}
	The proof is divided into several steps.\\
	
	{\bf {Step 1}}: For any $a<b$ in $[0,T]$, one first proves that $F$ is a mapping from $C^{\beta\text{-Hol}}([a-r,b],\R^d)$ into itself, or sufficiently 
	\[
	\ltn F(x)\rtn_{\beta,[a,b]}\leq \ltn I(x)\rtn_{\beta,[a,b]}+ \ltn J(x)\rtn_{\beta,[a,b]} <\infty. 
	\]
	With $a\leq s<t\leq b$, using assumption ($\text{\bf H}_f$) and assigning $L':=\max\{L_f,\|f(0)\|\}$ one has
	\begin{eqnarray*}
		\|I(x)(t)-I(x)(s)\|&=&\Big\|\int_s^tf(x_u)du\Big\|\\
		&\leq & L' (t-s)(1+\|x_.\|_{\infty,[s,t]})\\
		&\leq & L' (t-s)(1+\|x\|_{\infty,[a-r,b]})
	\end{eqnarray*}
	hence 
	\[
	\ltn I(x)\rtn_{\beta,[a,b]}\leq L' (b-a)^{1-\beta} (1+\|x\|_{\infty,[a-r,b]})\leq L' (b-a)^{1-\beta}(1+\|x\|_{\infty,\beta,[a-r,b]}) <\infty.
	\]
	On the other hand, using Lemma \ref{lem2}  with $K=\frac{1}{1-2^{1-(\nu+\beta)}}$, one has
	\begin{eqnarray*}
	&&	\|J(x)(t)-J(x)(s)\|\\
	&&=\Big|\int_s^tg(x_u)d\omega(u)\Big|\\
	&&\leq  \ltn \omega\rtn_{\nu,[a,b]}(t-s)^\nu\left(\|g(x_s)\| + K(t-s)^\beta\ltn g(x_.)\rtn_{\beta,[a,b]}\right)\\
	&&\leq \ltn \omega\rtn_{\nu,[a,b]}(t-s)^\nu\left(\|g(0)\|+ L_g\|x\|_{\infty,[a-r,b]} + L_gK(t-s)^\beta\ltn x\rtn_{\beta,[a-r,b]}\right),
	\end{eqnarray*}
	which implies 
	\begin{eqnarray}
	&&\ltn J(x)\rtn_{\beta,[a,b]}\notag\\
	&&\leq (b-a)^{\nu-\beta} \ltn \omega\rtn_{\nu,[a,b]}\left(\|g(0)\|+ L_g\|x\|_{\infty,[a-r,b]} + L_gK(b-a)^\beta\ltn x\rtn_{\beta,[a-r,b]}\right)\notag\\
	&&\leq (b-a)^{\nu-\beta} \ltn \omega\rtn_{\nu,[a,b]}\left[\|g(0)\|+ L_g+L_g  K(b-a)^\beta\right](1+\| x\|_{\infty,\beta,[a-r,b]})<\infty.\notag
	\end{eqnarray}
	Therefore $\ltn F(x)\rtn_{\beta,[a,b]}$ is finite.
	Moreover, by assigning $a:=t_0, b:= t_1$ it follows from the definition of $F$ that
	\begin{eqnarray}\label{F1}
	&& \| F(x)\|_{\infty,\beta,[t_0-r,t_1]}\notag\\
	&=& \|F(x)\|_{\infty,[t_0-r,t_1]}+\ltn F(x)\rtn_{\beta,[t_0-r,t_1]}\notag\\
	&\leq & \max\Big\{\|F(x)\|_{\infty,[t_0-r,t_0]},\|F(x)\|_{\infty,[t_0,t_1]} \Big\}+\ltn F(x)\rtn_{\beta,[t_0-r,t_0]} +\ltn F(x)\rtn_{\beta,[t_0,t_1]}\notag\\
	&\leq&\max\Big\{\|F(x)\|_{\infty,[t_0-r,t_0]},\|F(x)(t_0)\|+ (t_1-t_0)^\beta\ltn F(x)\rtn_{\beta,[t_0,t_1]}\Big\}\notag\\
	&&+\ltn F(x)\rtn_{\beta,[t_0-r,t_0]} +\ltn F(x)\rtn_{\beta,[t_0,t_1]}\notag\\
	&\leq &  \| x\|_{\infty,[t_0-r,t_0]}+ \ltn x\rtn_{\beta,[t_0-r,t_0]}+ [1+(t_1-t_0)^\beta]\ltn F(x)\rtn_{\beta,[t_0,t_1]}\notag\\
	&\leq & \| x\|_{\infty,\beta,[t_0-r,t_0]} + C'\left[(t_1-t_0)^{1-\beta} + (t_1-t_0)^{\nu-\beta} \ltn \omega\rtn_{\nu,[t_0,t_1]} \right](1+\|x\|_{\infty,\beta,[t_0-r,t_1]}),\notag\\
	\end{eqnarray}
	where 
	\begin{equation}\label{eqC}
	C'= C'(t_1-t_0) := [1+(t_1-t_0)^\beta](\|g(0)\|+ L_g + L_g K (t_1-t_0)^\beta+ L').
	\end{equation}
	Furthermore, for $0<\epsilon\leq \nu-\beta$ small enough, 
	\begin{eqnarray}\label{F2}
	&&\ltn F(x)\rtn_{(\beta+\epsilon),[t_0,t_1]}\notag\\
	&&\leq  C'\left((t_1-t_0)^{1-\beta-\epsilon} + (t_1-t_0)^{\nu-\beta-\epsilon} \ltn \omega\rtn_{\nu,[t_0,t_1]} \right)(1+\|x\|_{\infty,\beta,[t_0-r,t_1]}).
	\end{eqnarray}
	
	{\bf {Step 2}}:
	Following \cite{congduchong} and \cite{duc1} , assign
	\begin{equation}\label{eqCmax}
	C :=  2(\|g(0)\|+ L' + L_g (K+1))
	\end{equation}
	and fix $\mu < \min\{1,C\}$. We construct a sequence $t_i$ in $[0,\infty)$ such that $t_0=0$ and 
	\[
	t_{i+1}= \sup \{t\geq t_i : C\left[(t-t_i)^{1-\beta} + (t-t_i)^{\nu-\beta} \ltn \omega\rtn_{\nu,[t_i,t]} \right]\leq\mu\}.
	\]
	Since $\omega \in C^{0,\nu\rm{-Hol}}([0,T],\R)$,
	\[
	\Big|\ltn \omega\rtn_{\nu,[0,\tau]} - \ltn \omega\rtn_{\nu,[0,\tau\pm h]}\Big| \leq \max \Big\{\ltn \omega\rtn_{\nu,[\tau,\tau+h]}, \ltn \omega\rtn_{\nu,[\tau-h,\tau]} \Big\} \to 0 \text{\ as\ } h \to 0^+,
	\]
	the function $\tau^{1-\beta} + \tau^{\nu-\beta}\ltn \omega\rtn_{\nu,[0,\tau]}$ is then continuous due to the continuity of each component in $\tau$. Hence 
	\begin{equation}\label{stoppingtime}
	(t_{i+1} - t_i)^{1-\beta} + (t_{i+1} - t_i)^{\nu-\beta}   \ltn \omega\rtn_{\nu,[t_i,t_{i+1}]} = \frac{\mu}{C},\ \forall i \geq 0.
	\end{equation}
	If $t_\infty:=\sup t_i < \infty$, then by choosing $k$ such that $k(\nu-\beta)\geq 1$, one has
	\begin{eqnarray*}
		n(\mu/C)^k&\leq& \sum_{i=0}^{n-1} \left[ (t_{i+1}-t_i)^{1-\beta} + (t_{i+1}-t_i)^{\nu-\beta} \ltn \omega\rtn_{\nu,[t_i,t_{i+1}]}\right]^k\\
		&\leq &  2^{k-1}\sum_{i=0}^{n-1} \left[ (t_{i+1}-t_i)^{k(1-\beta)} + (t_{i+1}-t_i)^{k(\nu-\beta)} \ltn \omega\rtn^k_{\nu,[0,t_\infty]}\right]\\
		&\leq &  2^{k-1} \left[ \sum_{i=0}^{n-1}(t_{i+1}-t_i)^{k(1-\beta)} + \sum_{i=0}^{n-1}(t_{i+1}-t_i)^{k(\nu-\beta)} \ltn \omega\rtn^k_{\nu,[0,t_\infty]}\right]\\
		&\leq &  2^{k-1} t_\infty^{k(1-\beta)} + t_\infty^{k(\nu-\beta)} \ltn \omega\rtn^k_{\nu,[0,t_\infty]}< \infty
	\end{eqnarray*}
	for all $n\in \N$, which is contradiction. Hence $\{t_i\}$ is increasing to infinity and it makes sense to define 
	\[
	N(T,\omega) := \max \{i: t_i \leq T\}.
	\]
	Moreover,
	\begin{eqnarray}\label{NT}
	N(T,\omega)\leq 2^{k-1}\left(\frac{C}{\mu}\right)^k \left( T^{k(1-\beta)} + T^{k(\nu-\beta)} \ltn \omega\rtn^k_{\nu,[0,T]}\right).
	\end{eqnarray}
	
	{\bf {Step 3}}: In this step one shows the {\em local existence} of solution on $[t_0,t_1]$ constructed as above. From definition of stopping times, $|t_1-t_0|<1$ and  $C'(t_1-t_0)\leq C$, hence it follows that
	\[
	F: C^{\beta\text{-Hol}}([t_0-r,t_1],\R^d)\rightarrow C^{\beta\text{-Hol}}([t_0-r,t_1],\R^d)
	\]
	satisfying
	\begin{equation}\label{solutionmap}
	\| F(x)\|_{\infty,\beta,[t_0-r,t_1]}\leq \| x\|_{\infty,\beta,[t_0-r,t_0]} + \mu (1+\|x\|_{\infty,\beta,[t_0-r,t_1]})
	\end{equation}
	Introducing the set 
	\[
	B:= \left\{x\in C^{\beta\text{-Hol}}([t_0-r,t_1],\R^d)|\;x_{t_0} = \eta, \|x\|_{\infty,\beta,[t_0-r,t_1]}\leq R:=\frac{ \| \eta\|_{\infty,\beta,[t_0-r,t_0]}+\mu}{1-\mu} \right\},
	\]
	then $F:B\rightarrow B$. %The following lemma is needed to prove the continuity of $F$.
	%%%%%%%%%%%%%%%%%%%%%%%%%%%%%%%%%%%%%%%55
	%Next, for $x,y\in B$,
	By Lemma \ref{lem4} and the definition of $F$, the following estimate
	\begin{eqnarray*}
		\|F(x)-F(y)\|_{\infty,\beta,[t_0-r,t_1]}&=&\|F(x)-F(y)\|_{\infty,[t_0,t_1]}+\ltn F(x)-F(y)\rtn_{\beta,[t_0,t_1]}\\
		&\leq& \Big[1+(t_1-t_0)^\beta\Big]\ltn F(x)-F(y)\rtn_{\beta,[t_0,t_1]}\\
		&\leq& L(t_1-t_0,R)\Big[1+(t_1-t_0)^\beta\Big]\|x-y\|_{\infty,\beta,[t_0-r,t_1]}.
	\end{eqnarray*}
	proves the continuity of $F$ on $B$.\\
	Observe that $F$ is a compact operator on $B$. Indeed, take the sequence $y^n=F(x^n)$, $x^n\in B$, by \eqref{F2}     
	$$\ltn y^n\rtn_{(\beta+\epsilon),[t_0,t_1]}\leq C \left((t_1-t_0)^{1-\beta-\epsilon} + (t_1-t_0)^{\nu-\beta-\epsilon} \ltn \omega\rtn_{\nu,[t_0,t_1]} \right)(1+R).$$
	By Proposition 5.28 of \cite{friz}, there exists a subsequence $y^{n_k}1_{[t_0,t_1]}$ which converges in $C^{\beta\text{-Hol}}([t_0,t_1],\R^d)$. Additionally, for all $k$, $y^{n_k}(t) = \eta(t), \forall t\in [t_0-r,t_0], $ hence
	\begin{eqnarray*}
		\|y^{n_k}-y^{n_{k'}}\|_{\infty,\beta,[t_0-r,t_1]}=\|y^{n_k}-y^{n_{k'}}\|_{\infty,\beta,[t_0,t_1]} \to 0 \;\; \text{as}\;\; k,k'\to \infty.
	\end{eqnarray*}
	Since $C^{\beta\text{-Hol}}([t_0-r,t_1],\R^d)$ is Banach one concludes that there is a subsequence of $y^n$ that converges in $C^{\beta\text{-Hol}}([t_0-r,t_1],\R^d)$. \\
	To sum up, $F:B\rightarrow B$ is a compact operator on the closed ball of Banach space $C^{\beta\text{-Hol}}([t_0-r,t_1],\R^d)$. By Schauder-Tychonoff fixed point  theorem (see e.g \cite[Theorem 2.A, p. 56]{Zeidler}),  there exists a function $x^*\in B$ such that $F(x^*)=x^*$, i.e $x^*$ is a local solution of \eqref{Eq2} on $[t_0-r,t_1]$. The fact that $x^* \in C^{\nu \text{-Hol}}([t_0-r,t_1],\R^d)$ is then obvious.  \\
	
	{\bf {Step 4}}: The local solution is unique.\\
	Assuming that $x $ and $y$ are solutions to \eqref{Eq2} on $[t_0-r,t_1]$ with the same initial condition $\eta$ , bounded by $M>0$. Put $z=x-y$ then $F(x)-F(y) = z$. By virtue of Lemma \ref{lem4}, for $t_0\leq s<t\leq t_1$,
	\begin{eqnarray*}
		\ltn z\rtn_{\beta,[s,t]} &\leq &L(t_1-t_0,M) \left[(t-s)^{1-\beta} +(t-s)^{\nu-\beta} \ltn \omega\rtn_{\nu,[s,t]}\right] \| z\|_{\infty,\beta,[s-r,t]}\notag\\
		&\leq &L(t_1-t_0,M) \left[(t-s)^{1-\beta} +(t-s)^{\nu-\beta} \ltn \omega\rtn_{\nu,[s,t]}\right] \left(\| z\|_{\infty,[s-r,t]}+\ltn z\rtn_{\beta,[s-r,t]}\right)\notag\\
		&\leq &L(t_1-t_0,M) \left[(t-s)^{1-\beta} +(t-s)^{\nu-\beta} \ltn \omega\rtn_{\nu,[s,t]}\right] \\
		&&\times\left(\max\{\| z\|_{\infty,[s-r,s]},\| z\|_{\infty,[s,t]} \}+\ltn z\rtn_{\beta,[s-r,s]}+\ltn z\rtn_{\beta,[s,t]}\right).\notag
	\end{eqnarray*}
	Since $| z\|_{\infty,[s,t]}\leq \|z(s)\| + (t-s)^\beta \ltn z\rtn_{\beta,[s,t]}\leq \| z\|_{\infty,[s-r,s]} + \ltn z\rtn_{\beta,[s,t]}$, it follows that
	\begin{eqnarray}\label{step4eq1}
	&&\ltn z\rtn_{\beta,[s,t]} \notag\\
	&&\leq 2L(t_1-t_0,M) \left[(t-s)^{1-\beta} +(t-s)^{\nu-\beta} \ltn \omega\rtn_{\nu,[s,t]}\right]\left( \| z\|_{\infty,\beta,[s-r,s]} + \ltn z\rtn_{\beta,[s,t]}\right). \notag\\
	\end{eqnarray}
	Construct similarly to {\bf {Step 2}} a finite sequence $\{s_i\}$ on $[t_0,t_1]$ such that $s_0=t_0$ and
	\[
	(s_{i+1} - s_i)^{1-\beta} + (s_{i+1} - s_i)^{\nu-\beta}   \ltn \omega\rtn_{\nu,[s_i,s_{i+1}]} = \frac{\mu}{2L(t_1-t_0,M)}.
	\]
	It follows from \eqref{step4eq1} that 
	\begin{eqnarray}
	\ltn z\rtn_{\beta,[s_0,s_1]} \leq \mu \left( \| z\|_{\infty,\beta,[s_0-r,s_0]} + \ltn z\rtn_{\beta,[s_0,s_1]}\right)=\mu \ltn z\rtn_{\beta,[s_0,s_1]}.
	\end{eqnarray}
	Consequently,  $\ltn z\rtn_{\beta,[s_0,s_1]}=0$. By induction, one can prove that $\ltn z\rtn_{\beta,[t_0,t_1]}=0$. Therefore, $z(u)\equiv 0,\forall u\in [t_0-r,t_1]$, i.e. $x\equiv y$ on $[t_0-r,t_1]$.\\

	{\bf {Step 5}}: By induction, there exists a unique solution of \eqref{Eq2} on each $[t_i-r,t_{i+1}]$. Finally, due to the unboundedness of $\{t_i\}$ the solution of \eqref{Eq2} can be extended to the whole $[-r,T]$ by concatenation. \\
\end{proof}

%%%%%%%%%%%%%%%%%%%%%%%%%%%%%%%%%%%%%%%

\begin{theorem}\label{thm2}
	Under the assumptions of Theorem \ref{thm1}, one has
	\begin{equation}\label{thm2eq0}
	\sup_{t \in [t_{N(t,\omega)},t_{N(t,\omega)+1}]}\|x_t\|_{\infty,\beta,[-r,0]} \leq e^{-[N(t,\omega)+1]\log(1-\mu)} \Big[\|x_{t_0}\|_{\infty,\beta,[-r,0]} + 1 \Big],
	\end{equation}			
	where $N(t,\omega)$-the number of stopping times \eqref{stoppingtime} in $(0,t]$, can be approximated by \eqref{NT}. 
\end{theorem}

%%%%%%%%%%%%%%%%%%%%%%%%%%%%%%%%%%%%%%%

\begin{proof}
	From the proof of Theorem \ref{thm1}, in particular \eqref{F1} and \eqref{solutionmap}, it follows that for any $i \geq 0$
	\[
	\|x\|_{\infty,\beta,[t_i-r,t]} \leq \|x\|_{\infty,\beta,[t_i-r,t_i]} + \mu (1+ \|x\|_{\infty,\beta,[t_i-r,t]}),\ \forall t \in [t_i,t_{i+1}].
	\]
	In other words,
	\begin{equation}\label{thm2eq1}
	\|x\|_{\infty,\beta,[t_i-r,t]} \leq \frac{\mu}{1-\mu} + \frac{1}{1-\mu} \|x\|_{\infty,\beta,[t_i-r,t_i]},\ \forall t\in [t_i,t_{i+1}].
	\end{equation}
	On the other hand,
	\begin{eqnarray}\label{thm2eq2}
	\|x\|_{\infty,\beta,[t-r,t]} = \|x\|_{\infty,[t-r,t]} + \ltn x \rtn_{\beta,[t-r,t]} \leq \|x\|_{\infty,\beta,[t_i-r,t]},\ \forall t \in [t_i,t_{i+1}].
	\end{eqnarray}
	Hence it follows from \eqref{thm2eq1} and \eqref{thm2eq2} that
	\begin{equation*}
	\|x\|_{\infty,\beta,[t-r,t]} \leq \frac{\mu}{1-\mu}+ \frac{1}{1-\mu}\|x\|_{\infty,\beta,[t_i-r,t_i]},\ \forall t\in [t_i,t_{i+1}].
	\end{equation*}
	which implies that
	\begin{equation}\label{thm2eq3}
	\sup_{t \in [t_i,t_{i+1}]}\|x_t\|_{\infty,\beta,[-r,0]} \leq \frac{\mu}{1-\mu} + \frac{1}{1-\mu} \|x_{t_i}\|_{\infty,\beta,[-r,0]}. 
	\end{equation}
	In particular, for any $i\geq 0$,
	\[
	\|x_{t_{i+1}}\|_{\infty,\beta,[-r,0]} \leq \frac{\mu}{1-\mu} + \frac{1}{1-\mu} \|x_{t_i}\|_{\infty,\beta,[-r,0]},
	\]
	or equivalently
	\[
	\|x_{t_{i+1}}\|_{\infty,\beta,[-r,0]} + 1 \leq \frac{1}{1-\mu} \Big[\|x_{t_{i}}\|_{\infty,\beta,[-r,0]} + 1 \Big].
	\]
	By induction arguments, one can conclude that
	\begin{equation}\label{thm2eq4}
	\|x_{t_i}\|_{\infty,\beta,[-r,0]} \leq \Big[\frac{1}{1-\mu}\Big]^i \Big[\|x_{t_0}\|_{\infty,\beta,[-r,0]} + 1\Big] -1,\ \forall i\geq 0.
	\end{equation}
	\eqref{thm2eq0} is then a direct consequence of \eqref{thm2eq3} and \eqref{thm2eq4}. \\
\end{proof}
The arguments in the proof of Theorem \ref{thm2} help us to derive a type of Gronwall lemma for H\"older norms.

%%%%%%%%%%%%%%%%%%%%%%%%%%%%%%%%%%%%%%%

\begin{lemma}\label{lem5}{\bf [Gronwall-type lemma]}
	Assume that $z:[-r,T]\to\R^d$ satisfies for any $0\leq s \leq t \leq T$
	\begin{eqnarray}\label{basicassumption} 
	\ltn z\rtn_{\beta,[s,t]}\leq A+C\left((t-s)^{1-\beta}+(t-s)^{\nu-\beta}\ltn\omega\rtn_{\nu,[s,t]}\right)\| z\|_{\infty,\beta,[s-r,t]}
	\end{eqnarray}
	with some constants $A,C>0$. Then for $\mu<\min\{\frac{1}{2},C\}$  the following estimate holds %\textcolor{red}{$\mu<\frac{1}{2}$}
	\begin{equation}\label{lem5eq1}
	\|z_t\|_{\infty,\beta,[-r,0]}\leq e^{-[N(t,\omega)+1]\log(1-2\mu)}\Big[\frac{A}{\mu} +\|z\|_{\infty,\beta,[-r,0]} \Big],\ \forall t\in [0,T]. 
	\end{equation}
\end{lemma}

%%%%%%%%%%%%%%%%%%%%%%%%%%%%%%%%%%%%%%%

\begin{proof} 
	Using the construction of stopping times in \eqref{stoppingtime}, one has
	\[
	\ltn z \rtn_{\beta,[t_i,t]} \leq A + \mu \|z\|_{\infty,\beta,[t_i-r,t]},\ \forall t \in [t_i,t_{i+1}], 
	\]
	hence
	\begin{eqnarray*}
		\|z\|_{\infty,\beta,[t_i-r,t]} &\leq& \max\{\|z\|_{\infty,[t_i-r,t_i]}, \| z \|_{\infty,[t_i,t]}\} + \ltn z \rtn_{\beta,[t_i-r,t_i]}+\ltn z \rtn_{\beta,[t_i,t]} \\
		&\leq&  \|z\|_{\infty,\beta,[t_i-r,t_i]} + (1+(t_{i+1}-t_i)^\beta) \ltn z\rtn_{\beta,[t_i,t]}\\
		&\leq& \|z\|_{\infty,\beta,[t_i-r,t_i]} + 2 [A + \mu \|z\|_{\infty,\beta,[t_i-r,t]}] 
	\end{eqnarray*}
	due to the fact that $\frac{\mu}{C}<1$. It follows that %  \textcolor{red}{ $\mu<1$}
	\begin{eqnarray*}
		\|x\|_{\infty,\beta,[t_i-r,t]} \leq \frac{2A}{1-2\mu} + \frac{1}{1-2\mu}\| z \|_{\infty,\beta,[t_i-r,t_i]},\ \forall t \in [t_i,t_{i+1}],
	\end{eqnarray*}
	(provided that $\mu<\frac{1}{2}$), which has similar form to \eqref{thm2eq1}. As a consequence, by following the same arguments as in Theorem \ref{thm2}, one has
	\begin{equation*}
	\|x_{t_i}\|_{\infty,\beta,[-r,0]} \leq \Big[\frac{1}{1-2\mu}\Big]^i \Big[\|x_{t_0}\|_{\infty,\beta,[-r,0]} + \frac{A}{\mu}\Big] -\frac{A}{\mu},\ \forall i\geq 0,
	\end{equation*}
	which proves \eqref{lem5eq1}.\\
\end{proof}

%%%%%%%%%%%%%%%%%%%%%%%%%%%%%%%%%%%%%%%
Denote by $x(\cdot,\omega,\eta)$ the solution of \eqref{Eq1} with initial function $\eta$. We prove in the following the continuity of the solution with respect to the initial condition.

%%%%%%%%%%%%%%%%%%%%%%%%%%%%%%%%%%%%%%%

\begin{theorem}\label{thm3}
	Under the assumptions of Theorem \ref{thm1}, the solution $x_t(\cdot,\omega,\eta)$ is continuous with respect to $\eta$.
\end{theorem}
\begin{proof}
	For any $\eta^1, \eta^2 \in C^{\beta\text{-Hol}}([-r,0],\R^d)$ denote $x^i(\cdot)=x(\cdot,\omega,\eta^i)$, $i=1,2$. Fix $\eta^1$, by \eqref{thm2eq0} one can choose $M$ large enough such that $\|x(\cdot,\omega,\eta^2)\|_{\infty,\beta,[-r,T]}\leq M$ for all $\eta^2$ such that $\|\eta^2 -\eta^1\|_{\infty,\beta,[-r,0]}\leq 1$. From \eqref{lem4eq1} in Lemma \ref{lem4}, one has for all $0\leq a\leq b\leq T$,
	\begin{eqnarray*}
		\ltn x^1 - x^2\rtn_{\beta,[a,b]} \leq  L(T,M)\left((b-a)^{1-\beta}+ (b-a)^{\nu-\beta}\ltn\omega\rtn_{\nu,[a,b]}\right) \|x^1-x^2\|_{\infty,\beta,[a-r,b]},
	\end{eqnarray*}
	which has the form \eqref{basicassumption} with $A=0$ and $C= L(T,M)$.
	Therefore,
	\[
	\|x_t(\cdot,\omega,\eta^2) - x_t(\cdot,\omega,\eta^1)\|_{\infty,\beta,[-r,0]}\leq e^{-[N(t,\omega)+1]\log(1-2\mu)} \|\eta^1 - \eta^2\|_{\infty,\beta,[-r,0]}, \forall t\in [0,T],
	\]
	in which $N(t,\omega)$ is defined in \eqref{NT} with $C= L(T,M)$ and $\mu < \min\{1/2, L(T,M)\}$
	and $N$ depends on $L(T,M)$ - the local constant in the vicinity of $\eta$. That proves the continuity of $x_t(\cdot,\omega,\eta)$ w.r.t. the initial function $\eta$.
\end{proof}

\begin{remark}
	It can be seen that  for $x\in C^{\beta\rm{-Hol}}([-r,T],\R^d)$ there exists $C(T,r)$ such that
	$$\|x(\cdot)\|_{\infty,\beta,[-r,T]} \leq C(T,r)\sup_{t\in[0,T]}\|x_t(\cdot)\|_{\infty,\beta,[-r,0]}.$$
	Indeed, since $\|x(\cdot)\|_{\infty,[-r,T]} =\sup_{t\in[0,T]}\|x_t(\cdot)\|_{\infty,[-r,0]}$, for $s,t \in[-r,T]$ one can construct a finite sequence $s_i$ as follow: $s=s_0, s_1=s_0+r, s_2=s_1+r,\dots $, until $s_n+r\geq t$ and assign $s_{n+1}:=t$. Then
		\begin{eqnarray*}
			\frac{\|x(t)-x(s)\|}{|t-s|^\beta} &\leq & \sum_{i=0}^{n}\frac{\|x(s_{i+1})-x({s_i})\|}{|s_{i+1}-s_i|^\beta}\\
			&\leq &  \sum_{i=0}^{n}\frac{\|x_{s_{i+1}}(s_i-s_{i+1})-x_{s_{i+1}}(0)\|}{|s_{i+1}-s_i|^\beta}\\
			&\leq &  \sum_{i=0}^{n}\ltn x_{s_{i+1}}\rtn_{\beta,[-r,0]}\\
			&\leq & (1+T/r)\sup_{t\in[0,T]}\ltn x_{t}\rtn_{\beta,[-r,0]}.
		\end{eqnarray*}
	Hence, from Theorem \ref{thm3} one concludes that 
	\[
	\|x(\cdot,\omega,\eta^2) - x(\cdot,\omega,\eta^1)\|_{\infty,\beta,[-r,T]}\leq C(T,r)e^{-[N(T,\omega)+1]\log(1-2\mu)} \|\eta^1 - \eta^2\|_{\infty,\beta,[-r,0]}.
	\]

\end{remark}
%%%%%%%%%%%%%%%%%%%%%%%%%%%%%%%%%%%%%%%

Next, assuming that $f$ is $C^1$, we fix a solution $x(\cdot,\omega,\eta)$ of \eqref{Eq1} and consider the linearized equation
\begin{eqnarray}\label{linear.equa}
y(t)=\eta^1(0)-\eta(0)+\int_0^t Df(x_s)y_sds+\int_0^tDg(x_s)y_sd\omega(s),
\end{eqnarray}
with initial function $\eta^1-\eta\in C^{\beta\rm{-Hol}}([-r,0],\R^d)$. Since $y\in C^{\beta\text{-Hol}}([-r,T],\R^d)$ and
\begin{eqnarray}\label{Dg}
\|Dg(x_t)-Dg(x_s)\|&\leq & L_M \|x_t-x_s\|^\delta_{\infty,[-r,0]}\notag\\
&\leq& L_M  \ltn x\rtn_{\beta,[-r,T]}^\delta (t-s)^{\delta\beta}\notag\\
\ltn Dg(x_.)\rtn_{\delta\beta,[a,b]} &\leq & L_MM^\delta
\end{eqnarray}
with $M\geq\| x\|_{\infty,\beta,[-r,T]}$, the integrals $\int_0^t Df(x_s)y_sds $ and $\int_0^tDg(x_s)y_sd\omega(s)$ are well defined. We need to prove the following lemma
\begin{lemma}
	The equation \eqref{linear.equa} has unique solution  $y$ in $C^{\delta\beta\rm{-Hol}}([-r,T],\R^d)$. Moreover, the solution is H\"older continuous with exponent $\nu$ on $[0,T]$.
\end{lemma}
\begin{proof}
	The proof is similar to that of Theorem \ref{thm1}. Note that $\|Df(\xi)\|\leq L_f$ for all $\xi\in C_r$ \\
	Define the map
	\begin{eqnarray*}
		G(y)(t)=
		\begin{cases}
			&  y(t_0)+ \int_0^t Df(x_s)y_sds+\int_0^tDg(x_s)y_sd\omega(s),\quad \text{if}\;\; t\in [t_0,t_1]\\
			& y(t), \qquad \qquad \qquad \qquad \qquad \qquad \qquad \qquad \quad\;\;\; \text{if}\;\; t\in [t_0-r,t_0],
		\end{cases}
	\end{eqnarray*}	
	then for $s,t\in [0,T]$
	\begin{eqnarray*}
		&&\|G(y)(t)-G(y)(s)\|\\
		&\leq & L_f \|y\|_{\infty,[s-r,t]}(t-s)+ \ltn\omega\rtn_{\nu,[s,t]}(t-s)^\nu \Big[L_g \|y\|_{\infty,[s-r,t]}\\
		&&+ K' (t-s)^{\delta\beta}\|y\|_{\infty,[s-r,t]}\ltn Dg(x_.)\rtn_{\delta\beta,[s,t]}+K'L_g (t-s)^{\delta\beta} \ltn y\rtn_{\delta\beta,[s-r,t]}\Big],
	\end{eqnarray*}
	with $K'=\frac{1}{1-2^{1-(\nu+\delta\beta)}}$. Combining with \eqref{Dg}, it follows that
	\begin{eqnarray*}
		\ltn G(y)\rtn_{\delta\beta,[s,t]}&\leq & C \left((t-s)^{1-\delta\beta} + \ltn\omega\rtn_{\nu,[s,t]}(t-s)^{\nu-\delta\beta} \right) \| y\|_{\delta\beta,[s-r,t]}
	\end{eqnarray*}
	Repeat the arguments in Theorem \ref{thm1}, one can prove the existence of solution to \eqref{Dg}. Since $G$ is linear, the uniqueness of the solution is derived by a contraction mapping argument. Finally, it is obvious that the solution depends linearly on the initial function.
\end{proof}

% % % % % % % % % % % % % % %cbbvberxef3segtderyfh56uj67i

\begin{theorem}\label{thm4}
	Assuming that $f,g$ satisfy conditions $(H_f)$ and $(H_g)$ and $f$ is a $C^1$ function. Then the solution $x_t(.,\omega,\eta)$ of \eqref{Eq2} is differentiable with respect to initial function $\eta$.
\end{theorem}

\begin{proof}
	Consider two solutions $x(\cdot)=x(\cdot,\omega,\eta)$ and $x^1(\cdot)=x(\cdot,\omega,\eta^1)$ of  \eqref{Eq2}
	\begin{eqnarray*}
		x^1(t) &=& \eta^1(0)+\int_0^t f(x_s^1)ds +\int_0^tg(x^1_s)d\omega(s)\\
		x(t) &=& \eta(0)+\int_0^t f(x_s)ds +\int_0^tg(x_s)d\omega(s),
	\end{eqnarray*}
	and the solution $y(\cdot)=y(\cdot,\omega,\eta^1-\eta)$ of \eqref{linear.equa}. Define 
	\[
	z(\cdot)=x^1(\cdot)-x(\cdot)- y(\cdot)
	\]
	then $z\equiv 0$ on $[-r,0]$. By the assumptions, there exists $F^*,G^*$- the nonlinear remaining terms of $f,g$ such that 
	\begin{eqnarray*}
		f(x_s^1)-f(x_s) &=& Df(x_s)(x_s^1-x_s) + F^*(x^1_s-x_s) \\
		g(x_s^1)-g(x_s) &=& Dg(x_s)(x_s^1-x_s) + G^*(x^1_s-x_s).
	\end{eqnarray*}
	Since $f, g$ are $C^1$, there exist a number $h>0$ and continuous functions $p,\; q: [0,h] \to \R_+, p(0)=q(0)=0$ and $\lim \limits_{u \to 0} p(u) =\lim \limits_{u \to 0} q(u) =0$, such that 
	\begin{eqnarray*}
		\|F^*(x^1_s-x_s)\| &=& \Big\|\int_0^1 [Df(\theta x^1_s +(1-\theta) x_s) - Df(x_s)](x^1_s-x_s)d\theta \Big\| \\
		&\leq& \|x^1(\cdot)-x(\cdot)\|_{\infty,\beta,[-r,T]}\  p\big(\|x^1(\cdot)-x(\cdot)\|_{\infty,\beta,[-r,T]}\big)
	\end{eqnarray*}
	and
	\begin{eqnarray*}
		\|G^*(x^1_s-x_s)\| &=& \Big\|\int_0^1 [Dg(\theta x^1_s +(1-\theta) x_s) - Dg(x_s)](x^1_s-x_s)d\theta \Big\| \\
		&\leq& \|x^1(\cdot)-x(\cdot)\|_{\infty,\beta,[-r,T]} \ q\big(\|x^1(\cdot)-x(\cdot)\|_{\infty,\beta,[-r,T]}\big).
	\end{eqnarray*}
	whenever $\|x^1(\cdot)-x(\cdot)\|_{\infty,\beta,[-r,T]} \leq h$. %It follows that there exists a constant $C(F^*)$ such that
	%\begin{equation}\label{CF*}
	%\|F^*(x^1-x)\|_{\infty,[-r,T]} \leq C(F^*) \|x^1(\cdot)-x(\cdot)\|_{\infty,\beta,[-r,T]}\ p\Big(\|x^1(\cdot)-x(\cdot)\|_{\infty,\beta,[-r,T]}\Big).
	%\end{equation}
	Similar to Lemma \ref{lem3}, we estimate the H\"older norm of $G^*$. Specifically, for $0\leq s<t\leq T$, 
	\begin{eqnarray}\label{Gstar}
	&&\|G^*(x^1_t-x_t)-G^*(x^1_s-x_s)\|\notag\\
	&&=\left\| \int_0^1[Dg(\theta x^1_t+(1-\theta)x_t) -Dg(x_t)](x^1_t-x_t)d\theta \right.\notag\\
	&&\qquad \left.- \int_0^1[Dg(\theta x^1_s+(1-\theta)x_s) -Dg(x_s)](x^1_s-x_s)d\theta \right\|   \notag \\
	&&\leq  \left\| \int_0^1[Dg(\theta x^1_t+(1-\theta)x_t) -Dg(x_t)- Dg(\theta x^1_s+(1-\theta)x_s) \right.\notag\\
	&&\qquad +Dg(x_s)](x^1_t-x_t)d\theta \Big\|\notag \\
	&& +\left\| \int_0^1[Dg(\theta x^1_s+(1-\theta)x_s) -Dg(x_s)](x^1_t-x_t - x^1_s+x_s)d\theta \right\|.
	\end{eqnarray}
	From the assumption of $g$ the second integral in  \eqref{Gstar}  is less than or equal 
	$$
	L_M\|x^1_t-x_t - x^1_s+x_s\| .\|x^1_s-x_s\|^\delta_{\infty,[-r,0]},
	$$
	where $M$ is a upper bound of $\|x\|_{\infty,\beta,[-r,T]}$ and $\|x_1\|_{\infty,\beta,[-r,T]}$. % M can be choose = norm of x + h
	It follows from Lemma \ref{lem1} that
	\begin{eqnarray}\label{Gstar2}
	&&\left\| \int_0^1[Dg(\theta x^1_s+(1-\theta)x_s) -Dg(x_s)](x^1_t-x_t - x^1_s+x_s)d\theta \right\| \notag\\
	&&\leq   L_M(t-s)^\beta\|x^1-x\|^{1+\delta}_{\infty,\beta,[-r,T]}.
	\end{eqnarray}
	Since $\delta\beta+\nu>1$, one can choose $0<\gamma<1$ such that $\gamma\delta\beta+\nu>1$. Then
	\begin{eqnarray}
	&&\|Dg(\theta x^1_t+(1-\theta)x_t) -Dg(x_t)- Dg(\theta x^1_s+(1-\theta)x_s) +Dg(x_s)\|_{L(C_r,\R^d)}^\gamma   \notag    \\
	&&\leq  \|Dg(\theta x^1_t+(1-\theta)x_t) - Dg(\theta x^1_s-(1-\theta)x_s)\|_{L(C_r,\R^d)}^\gamma\notag\\
	&&\qquad +\|Dg(x_t) -Dg(x_s)\|_{L(C_r,\R^d)}^\gamma\notag\\
	&&\leq  2L_M^\gamma \left(\|x^1_t-x^1_s\|^{\gamma\delta}_{\infty,[-r,0]}+ |x_t-x_s\|^{\gamma\delta}_{\infty,[-r,0]}\right)\notag\\
	&&\leq  2L_M^\gamma (t-s)^{\gamma\delta\beta}\left(\|x^1_.\|^{\gamma\delta}_{\beta,[0,T]}+ \|x_.\|^{\gamma\delta}_{\beta,[0,T]}\right)\notag\\
	&&\leq  2L_M^\gamma (t-s)^{\gamma\delta\beta}\left(\|x^1\|^{\gamma\delta}_{\infty,\beta,[-r,T]}+ \|x\|^{\gamma\delta}_{\infty,\beta,[-r,T]}\right)\notag\\
	&&\leq  4M^{\gamma\delta}L_M^\gamma (t-s)^{\gamma\delta\beta}.
	\end{eqnarray} 
	On the other hand,
	\begin{eqnarray}
	&&\|Dg(\theta x^1_t+(1-\theta)x_t) -Dg(x_t)- Dg(\theta x^1_s+(1-\theta)x_s) +Dg(x_s)\|^{1-\gamma}_{L(C_r,\R^d)}\notag     \\
	&&\leq  \|Dg(\theta x^1_t+(1-\theta)x_t) - Dg(x_t)\|_{L(C_r,\R^d)}^{1-\gamma}\notag\\
	&&\quad +\|Dg(x^1_s+(1-\theta)x_s) -Dg(x_s)\|_{L(C_r,\R^d)}^{1-\gamma}\notag	\\
	&& \leq 2L_M^{1-\gamma}\|x^1-x\|^{(1-\gamma)\delta}_{\infty,\beta,[-r,T]}.
	\end{eqnarray}
	Therefore, the first integral in \eqref{Gstar} does not exceed   $8M^{\gamma\delta}L_M(t-s)^{\gamma\delta\beta}\|x^1-x\|^{1+(1-\gamma)\delta}_{\infty,\beta,[-r,T]}$. Combining this with \eqref{Gstar} and \eqref{Gstar2}, one obtains
	\begin{eqnarray*}
		\|G^*(x^1_t-x_t)-G^*(x^1_s-x_s)\|\leq   (t-s)^{\gamma\delta\beta}C(T,M)  \|  x^1-x\|^{1+(1-\gamma)\delta}_{\infty,\beta,[-r,T]},
	\end{eqnarray*}
	which implies 
	\begin{eqnarray}\label{CG*}
	\ltn G^*(x^1_.-x_.)\rtn_{\gamma\delta\beta,[0,T]} &\leq& C(G^*)\|x^1-x\|^{1+(1-\gamma)\delta}_{\infty,\beta,[-r,T]},
	\end{eqnarray}
	%\begin{equation}\label{CG*}
	%\|G^*(x^1-x)\|_{\infty,\beta,[-r,T]} \leq C(G^*) \|x^1(\cdot)-x(\cdot)\|^{1+\delta}_{\infty,\beta,[-r,T]}.
	%\end{equation}
Rewrite the equation of $z$ in the form
	\begin{eqnarray*}
		z(t)&=& \int_0^t\Big(f(x_s^1)-f(x_s)-Df(x_s)y_s\Big)ds+ \int_0^t\Big(g(x^1_s)-g(x_s)-Dg(x_s)y_s\Big)d\omega(s)\\
		&=& \int_0^t\Big(Df(x_s)z_s +F^*(x_s^1-x_s)\Big)ds + \int_0^t\Big(Dg(x_s)z_s +G^*(x_s^1-x_s)\Big)d\omega(s)\\
		&=& \Big[\int_0^t F^*(x_s^1-x_s) ds + \int_0^t G^*(x^1_s-x_s)d\omega(s)\Big]+\int_0^t Df(x_s)z_s ds + \int_0^t Dg(x_s)z_s d\omega(s).
	\end{eqnarray*}
	By similar estimates as in Theorem \ref{thm1}, there exist constants $K_1,K_2$ depending on $\nu,\beta,\delta,\gamma$ and generic constants $C_1,C_2$ such that for all $0\leq s<t\leq T$ %  depending on $F^*,G^*,T$ 
	\begin{eqnarray}\label{thm4eq1}
	\ltn z\rtn_{\beta,[s,t]} &\leq& \Big(|t-s|^{1-\beta}  \|F^*(x_\cdot^1-x_\cdot)\|_{\infty,[s,t]} \notag\\
	&& + |t-s|^{\nu-\beta} K(1+T^{\gamma\delta\beta}) \ltn \omega \rtn_{\nu,[s,t]} \|G^*(x^1_\cdot-x_\cdot)\|_{\infty,\gamma\delta\beta,[s,t]} \Big) \notag\\
	&&+ \Big( \|Df(x_\cdot)z_\cdot \|_{\infty,[s,t]} |t-s|^{1-\beta} \notag\\
	&&+ |t-s|^{\nu-\beta} K(1+T^{\delta\beta}) \ltn \omega \rtn_{\nu,[s,t]} \|Dg(x_\cdot)z_\cdot\|_{\infty,\delta\beta,[s,t]}\Big)\notag\\
	&\leq& C_1\Big(\|F^*(x^1-x)\|_{\infty,[-r,T]} + \|G^*(x^1-x)\|_{\infty,\gamma\delta\beta,[-r,T]}\Big)\notag\\
	&&+C_2\left((t-s)^{1-\beta}+(t-s)^{\nu-\beta}\ltn\omega\rtn_{\nu,[s,t]}\right)\|z\|_{\infty,\beta,[s-r,t]}\notag\\
	&\leq& C_1\Big(\|x^1-x\|_{\infty,\beta,[-r,T]}\  P(\|x^1-x\|_{\infty,\beta,[-r,T]})\Big)\notag\\
	&&+C_2\left((t-s)^{1-\beta}+(t-s)^{\nu-\beta}\ltn\omega\rtn_{\nu,[s,t]}\right)\|z\|_{\infty,\beta,[s-r,t]}
	\end{eqnarray}
	where $P(u) = p(u)+q(u) + u^{(1-\gamma)\delta}$. Due to Theorem \ref{thm3}, there exist constants $A(T,\eta), C(T,\eta)$, a number $h_1>0$  and a function $p_1: [0,h_1] \to \R_+$ with $p_1(0)=0$, $\lim \limits_{u \to 0} p_1(u) =0$, such that
	\begin{eqnarray*}
		\ltn z\rtn_{\beta,[s,t]} &\leq&  A(T,\eta)\|\eta^1 - \eta\|_{\infty,\beta,[-r,0]}  p_1\big(\|\eta^1 - \eta\|_{\infty,\beta,[-r,0]}\big) \\
		&&+ C(T,\eta) \left((t-s)^{1-\beta}+(t-s)^{\nu-\beta}\ltn\omega\rtn_{\nu,[s,t]}\right)\|z\|_{\infty,\beta,[s-r,t]},
	\end{eqnarray*}
	whenever $\|\eta^1 - \eta\|_{\infty,\beta,[-r,0]} \leq h_1$.
	Applying Lemma \ref{lem5}, one concludes that there exists a generic constant $C$ such that
	\begin{eqnarray*}
		\|z_.\|_{\infty,\beta,[0,T]} &\leq& C\Big[\|\eta^1 - \eta\|_{\infty,\beta,[-r,0]} p_1\big(\|\eta^1 - \eta\|_{\infty,\beta,[-r,0]} + \|z\|_{\infty,\beta,[-r,0]}\Big]\notag\\
		&\leq& C\|\eta^1 - \eta\|_{\infty,\beta,[-r,0]} p_1\big(\|\eta^1 - \eta\|_{\infty,\beta,[-r,0]}\big) 
	\end{eqnarray*}
	for all $\|\eta^1 - \eta\|_{\infty,\beta,[-r,0]} \leq h_1$, since $z \equiv 0$ on $[-r,0]$. Therefore, 
	\begin{eqnarray}\label{thm4eq2}
	&&\|x_t(\cdot,\omega,\eta^1)-x_t(\cdot,\omega,\eta)-y_t(\cdot,\omega,\eta^1-\eta)\|_{\infty,\beta,[-r,0]} \notag\\
	&&\leq \|\eta^1 - \eta\|_{\infty,\beta,[-r,0]} C p_1\big(\|\eta^1 - \eta\|_{\infty,\beta,[-r,0]}\big)
	\end{eqnarray}
	for any $\eta^1$ in the vicinity of $\eta$ such that $\|\eta^1 - \eta\|_{\infty,\beta,[-r,0]} \leq h_1$. Finally, \eqref{thm4eq2} implies that $x_t(\cdot,\omega,\eta)$ is differentiable with respect to $\eta$, with its derivative to be $y_t(\cdot,\omega,\cdot)$ . 
\end{proof}

\begin{acknowledgement}
We would like to thank the anonymous referees for their careful reading and insightful remarks which led to improvement of our manuscript.
This research is partly funded by Vietnam National Foundation for Science and Technology Development (NAFOSTED) and by Thang Long university. %101.03-2014.42. 
\end{acknowledgement}

\end{document}